\documentclass[11pt,reqno]{amsart}

\usepackage{amssymb,amsmath,graphicx,amsfonts,euscript}
\usepackage{color}

\numberwithin{equation}{section}

\def\H{{\cal H}}

\def\R{\mathbb{R}}

\def\H1{H^1(\R)}

\newtheorem{thm}{Theorem}
\newtheorem{lem}{Lemma}

\newtheorem{prop}{Proposition}

\makeatletter
\newcommand{\Extend}[5]{\ext@arrow0099{\arrowfill@#1#2#3}{#4}{#5}}
\makeatother

\begin{document}

\setcounter{page}{1}

\title[Solitary wave solutions for DNLS]{Orbital stability of solitary waves for derivative nonlinear Schr\"odinger equation}

\author{Soonsik Kwon}
\address{Department of Mathematical Sciences, Korea Advanced Institute of Science and Technology, 291 Daehak-ro, Yuseong-gu, Daejeon 34141, Korea.}
\email{soonsikk@kaist.edu}
\thanks{}

\author{Yifei Wu}
\address{Center for Applied Mathematics, Tianjin University, Tianjin 300072, P.R.China}
\email{yerfmath@gmail.com}
\thanks{S.K. is partially supported by NRF of Korea (No. 2015R1D1A1A01058832) and POSCO science fellowship. Y.W. is partially supported by the NSF of China (No. 11571118).}



\maketitle

\begin{abstract}\noindent
In this paper, we show the orbital stability of solitons arising in the cubic derivative nonlinear Schr\"odinger equations. We consider the zero mass case that is not covered by earlier works \cite{GuWu95,CoOh-06-DNLS}. As this case enjoys $L^2$ scaling invariance, we expect the orbital stability in the sense up to scaling symmetry, in addition to spatial and phase translations. For the proof, we are based on the variational argument and extend a similar argument in \cite{Wu2}. Moreover, we also show a self-similar type blow up criteria of solutions with the critical mass $4\pi$.
\end{abstract}

\section{Introduction}
We study the orbital stability of soliton solutions arising in the nonlinear Schr\"odinger equation with derivative (DNLS):
 \begin{equation}\label{eqs:DNLS}
   \left\{ \aligned
    &i\partial_{t}u+\partial_{x}^{2}u=i\partial_x(|u|^2u),\qquad t\in \R, x\in \R,
    \\
    &u(0,x)  =u_0(x)\in   H^1(\mathbb{R}).
   \endaligned
  \right.
 \end{equation}

The well-posedness for the equation \eqref{eqs:DNLS} is intensive studied. Especially, it was proved by Hayashi and Ozawa \cite{Ha-93-DNLS,HaOz-92-DNLS,HaOz-94-DNLS, Oz-96-DNLS} the local well-posedness in $\H1$ and  the global well-posedness when the initial data satisfies
$
\int_\R|u_0(x)|^2\,dx< 2\pi.
$
The results are analogous to that for the focusing quintic nonlinear Schr\"odinger equation. There are many low regularity local and global well-posedness results \cite{Ta-99-DNLS-LWP, Ta-01-DNLS-GWP, CKSTT-01-DNLS, CKSTT-02-DNLS, Herr, Grhe-95, NaOhSt-12, Miao-Wu-Xu:2011:DNLS}.
Recently, Wu \cite{Wu1,Wu2} showed that global well-posedness  holds as long as $
\int_\R|u_0(x)|^2\,dx< 4\pi$. In \cite{Wu2} the author observed that the threshold $4\pi$ corresponds to the mass of a ground state. This observation draws our attention to study the orbital stability or instability of soliton solutions with the critical mass $4\pi$.
\\
As is shown in \cite{GuWu95,CoOh-06-DNLS}, the equation \eqref{eqs:DNLS} has two parameter family of solitons of the form
$$ u_{\omega,c}(t,x) = \phi_{\omega,c} (x+ct)e^{i\omega t- i\frac{c}{2}(x+ct)+ \frac{3}{4}i\int_{-\infty}^{x+ct} |\phi_{\omega,c} (y)|^2\,dy},$$
where $(\omega,c) \in \R\times \R$, and $\phi_{\omega,c}$ is a ground state solution to the elliptic equation
\begin{equation}\label{eq:elliptic}
 -\partial_{xx} \phi + \big( \omega -\frac{c^2}{4})\phi +\frac{c}{2}\phi^3 -\frac{3}{16} \phi^5 =0, \end{equation}
If $c^2< 4\omega $, then $\phi_{\omega,c}$ shows an exponential decay:
$$ \phi_{\omega,c}(x) = \Big(\frac{\sqrt \omega}{4\omega -c^2}\big[\cosh(\sqrt{(4\omega-c^2) x}) -\frac{c}{\sqrt{4\omega}} \big]  \Big)^{-\frac 12}, $$
and the mass of $\phi_{\omega,c} $ is given by
$$ \|\phi_{\omega,c}\|_{L^2}^2  = 8 \tan^{-1} \sqrt{\frac{\sqrt{4\omega}+c }{\sqrt{4\omega}-c }}  <4\pi.$$
 The orbital stability of those solitons was proved in \cite{GuWu95} for $c<0$ and $c^2<4\omega$ and in \cite{CoOh-06-DNLS} for any $ c^2 <4\omega$. Here, the orbit is given by the phase and spatial translation. See \cite{Ohta-14} for the related studies.\\
In this work, we consider the endpoint case, $c^2=4\omega$. It is called \emph{the zero mass case} in view of \eqref{eq:elliptic}. Let $W$ be a ground state of the elliptic equation
\begin{align}\label{elliptic}
-\Phi_{xx}+\Phi^3-\frac{3}{16}\Phi^5=0.
\end{align}
Then, $ W_c(x) =c^{\frac 12} W(cx)$ is also the ground state solution to $$ -\Phi_{xx}+\frac{c}{2}\Phi^3-\frac{3}{16}\Phi^5=0, $$
and we have
$$ W(x) = 2^{\frac32}\big(4x^2+1\big)^{-\frac 12}, \qquad  \|W\|^2_{L^2} =\|W_c\|^2_{L^2}= 4\pi.$$
The corresponding solitary wave solution to \eqref{eqs:DNLS} with $4\pi$ mass is
\begin{align}\label{soliton}
R(t,x)=e^{\frac{3}{4}i\int_{-\infty}^{x+2t} |W(y)|^2\,dy} e^{-it-ix}W(x+2t).
\end{align}
We recall the mass, energy and momentum conservation laws:
\begin{align*}
M(u(t)) &=\int |u(t)|^2 \,dx, \\
E(u(t)) &= \int |u_x(t)|^2 +\frac 32 \text{Im} |u(t)|^2u(t) \overline{u_x(t)} +\frac 12 |u(t)|^6 \,dx, \\
P(u(t)) &= \text{Im} \int \overline{u(t)}u_x(t) -\frac 12 \int |u(t)|^4 \,dx.
\end{align*}
One may observe that $E(R)=P(R)=0$ and $M(R)=4\pi$.
Similarly, we denote $R_\lambda(t,x)=\lambda^{\frac12}R(\lambda^2 t, \lambda x)$. Then $R_\lambda$ is also a solution to \eqref{eqs:DNLS}. As opposed to the case of $ c^2 <4\omega $, the conservation laws do not restrict rescaling of solutions. Thus, our  main theorem of the \emph{conditional} orbital stability includes scaling parameter, in addition to the phase and spatial translation.
\begin{thm}\label{thm:main}
For any $\varepsilon>0$, there exists a $\delta=\delta(\varepsilon)$ such that if
\begin{align}\label{assump:intial}
\|u_0-R(0)\|_{H^1}\le \delta,
\end{align}
then for any $t\in I=(-T_*,T^*)$ (the maximal lifespan), either
$$
\|u(t)\|_{L^6}^6\le \sqrt\delta,
$$
or there exist $\theta(t)\in
[0,2\pi), y(t)\in\R$, and $\lambda(t)\in [ \lambda_0,\infty)$ for some constant $\lambda_0>0$, such that
$$
\|u(t)-e^{i\theta}R_\lambda(t,\cdot-y)\|_{H^1}\le \varepsilon.
$$
\end{thm}
\noindent The theorem implies that if initial data is close to the critical travelling wave, and the corresponding solution is inflation in some $L^p$ norm (ex. blowup in finite and infinite time), then the solution is also around the critical travelling wave solution up to the scaling, the phase and spatial translation. Hence, if $R(t,x)$ is instable, then it can be regarded as a type of geometry instability.

Moreover, from an extension of our argument we can also show a self-similar type blow-up criteria of solutions with the critical mass, which is equal to that of the ground state $W$.
\begin{thm}\label{thm:main2}
Let $u_0\in \H1$ with $\|u_0\|_{L^2}=\|W\|_{L^2}$. Suppose that the solution $u$ to \eqref{eqs:DNLS} blows up in the finite time $T^*$, then there exist $\theta(t)\in
[0,2\pi), y(t)\in\R$, such that when $t\to T^*$,
$$
e^{-i\theta(t)}u_{\lambda(t)}\big(t,\cdot+y(t)\big)-R(t)\to 0,\quad \mbox{ strongly in } \H1,
$$
where $u_\lambda(t,x)=\lambda^{\frac12}u(\lambda^2t, \lambda x)$, and $\lambda(t)=\|\partial_x W\|_{L^2}/\|\partial_x v(t)\|_{L^2}$.
\end{thm}

The proof of the theorems are based on the following variational result. Let the quantities
\begin{align}
S(w)=&\|w_x\|_{L^2}^2+\frac12\|w\|_{L^4}^4-\frac{1}{16}\|w\|_{L^6}^6;\\
K(w)=&6\|w\|_{L^4}^4-\|w\|_{L^6}^6.
\end{align}
We note that $K(W)=0$. Then we have the following rigidity of $W$.
\begin{prop}\label{prop:W-rigidity}
Let $g\in \H1$. For any $\varepsilon>0$, there exists $\varepsilon_0$, such that if
\begin{align}\label{condition-sckc}
\big|S(g)-S(W)\big|< \varepsilon_0,\quad K(g)\le 0,
\end{align}
then
$$
\inf\limits_{(\theta,y)\in\R^2} \big\|g-e^{i\theta}W(\cdot-y)\big\|_{\dot H^1}<\varepsilon.
$$
\end{prop}
\noindent We provide the proof of Proposition~\ref{prop:W-rigidity} in Section 2. We use the fact that  $W$ is  an optimal function of a sharp Gagliardo-Nirenberg inequality (see \cite{Agueh}),
\begin{equation}\label{sharp Gagliardo-Nirenberg2}
\|f\|_{L^6}\leq C_{GN} \|f\|_{L^4}^{\frac89}\|f_x\|_{L^2}^{\frac19},
\end{equation}
where we denoted $C_{GN}$ to be the sharp constant: $C_{GN}=3^{\frac16}(2\pi)^{-\frac19}$.
Roughly speaking, Proposition~\ref{prop:W-rigidity} tells that if a function closely attains the equality of the sharp Gagliardo-Nirenberg inequality \eqref{sharp Gagliardo-Nirenberg2}, then it is close to $W$ up to the symmetries of spatial, phase translation and scaling.\\
\\
The strategy to prove Theorem \ref{thm:main} and Theorem \ref{thm:main2}  is a variational argument. In addition, we combine it with the argument in \cite{Wu2}. To do this, we use the following gauge transformation. Let
\begin{equation}\label{gauge1}
v(t,x):=e^{-\frac{3}{4}i\int_{-\infty}^x |u(t,y)|^2\,dy}u(t,x),
\end{equation}
then from \eqref{eqs:DNLS}, $v$ is the solution of
 \begin{equation}\label{eqs:DNLS-under-gauge1}
   \left\{ \aligned
    &i\partial_{t}v+\partial_{x}^{2}v=\frac i2|v|^2v_x-\frac i2v^2\bar{v}_x-\frac
     3{16}|v|^4v,\qquad t\in \R, x\in \R,
    \\
    &v(0,x)  =v_0(x),
   \endaligned
  \right.
 \end{equation}
where $v_0(x):=e^{-\frac{3}{4}i\int_{-\infty}^x |u_0|^2\,dy}u_0$.
First, we use a similar argument in \cite{Wu2} to show that
$
f(t):=\|v(t)\|_{L^4}^4/\|v(t)\|_{L^6}^3,
$
is close to $\sqrt{\frac{8}3\pi}$. This almost fixes the ratio between $\|v(t)\|_{L^4}$ and $\|v(t)\|_{L^6}$. Then we use the conservation laws,
to establish the relationships between $\|v(t)\|_{L^4}, \|v(t)\|_{L^6}$  and $\|v_x(t)\|_{L^2}$. Then after suitable transformations, the solution almost
attains the equality of the sharp Gagliardo-Nirenberg \eqref{sharp Gagliardo-Nirenberg2}. Using Proposition \ref{prop:W-rigidity}, we conclude main theorems.\\
In Section 2, we prove Proposition~\ref{prop:W-rigidity} and in Section 3, we prove Theorem \ref{thm:main} and \ref{thm:main2}.

\section{Proof of Proposition \ref{prop:W-rigidity}}

First, we recall the uniqueness of the non-trivial solution for \eqref{elliptic}. Indeed, the non-trivial solution for \eqref{elliptic}, which vanishes at infinity, is uniqueness up to the rotation and the spatial transformations.
\begin{lem}\label{lem:uniqueness}
If $w\in  \H1\setminus \{0\}$ is a solution for  \eqref{elliptic}, then there exists $(\theta, x_0)$ such that
$$
w(x)=e^{i\theta}W(x-x_0).
$$
\end{lem}
\begin{proof}
See for example Berestycki and Lions \cite{BeLi-1983} for the standard argument.
\end{proof}
If $w$ is the solution of \eqref{elliptic}, we have  $K(w)=0$. Indeed, it follows from integrating against $\frac12w-x\partial_x w$ on the both side of
\eqref{elliptic} and then integration.
Furthermore, set
\begin{equation}\label{dc}
d:=\inf\{S(\phi):\phi\in \H1\setminus \{0\}, K(\phi)=0\}.
\end{equation}
Then $d\le S(W)$ due to $K(W)=0$. \\
Moreover, using the fact $K(\phi)=0$, we claim that $d> 0$. If we assume that $d=0$, then there exists a sequence $\{g_n\}_{n=1}^\infty \subset \H1\setminus \{0\}$, such that
$$
K(g_n)=0,\qquad \mbox{ and } \qquad S(g_n)\to 0.
$$
This gives
$$
-\frac1{12}K(g_n)+S(g_n)=\|\partial_x g_n\|_{L^2}^2+\frac1{48}\|g_n\|_{L^6}^6\to 0.
$$
Hence, by interpolation, there exists  $N_0$, such that $n\ge N_0$,
$$
\|g_n\|_{L^\infty}\le 1.
$$
Now by the definition of $K$, for $n\ge N_0$,
\begin{align}
0=K(g_n)&=\int \big(6|g_n|^4-|g_n|^6\big)\,dx=\int |g_n|^4\big(6-|g_n|^2\big)\,dx\notag\\
&\ge \int |g_n|^4\big(6-1\big)\,dx=5 \int |g_n|^4\,dx. \label{K-0}
\end{align}
That is, $\int  |g_n|^4\,dx=0$. This implies that $g_n\equiv 0$. This contradicts with $g_n\neq 0$. Hence, we conclude $d>0$. \\
\\ \noindent Next, we shall prove that $W$ is the unique minimizer (up to symmetries) which attains $d$. First of all, we prove the existence of the minimizer.
\begin{prop}\label{prop:SK-minimizer}
For any sequence $\{g_n\}\subset H^1(\R)$  satisfying that
$$
S(g_n)\to d,  \mbox{ as } n\to \infty,\quad \mbox{ and }\quad K(g_n)\le  0,
$$
there exists a function $G$, such that
$$
g_n\to G \quad \mbox{ in } \dot H^1(\R).
$$
In particular, $S(G)=d$, and $K(G)=0$.
\end{prop}
\begin{proof}
By the profile decomposition with respect to $H^1$ Sobolev embedding (see \cite{Garard-98}
for example), there exist sequences $\{V^j\}_{j=1}^\infty , \{x_n^j\}_{j,n=1}^\infty, \{R^L_n\} $
such that, up to a subsequence, for each $L$
\begin{equation}\label{profile-dep}
g_n =\sum\limits_{j=1}^L V^j(\cdot-x_n^j)+R_n^L,
\end{equation}
where $|x_n^j-x_n^k|\to \infty, \mbox{ as } n\to\infty, j\neq k$, and
\begin{equation}\label{Min-Norm-H1}
\aligned
  &\lim\limits_{L\to\infty}\left[\lim\limits_{n\to\infty}\|R_n^L\|_{L^4\cap L^6}\right]=0.
  \endaligned
\end{equation}
Moreover,
\begin{align}
  \|g_n\|_{L^4}^4&=\sum\limits_{j=1}^L \|V^j\|_{L^4}^4+\|R_n^L\|_{L^4}^4+o_n(1),\label{Sum-Norm-L4}\\
   \|g_n\|_{L^6}^6&=\sum\limits_{j=1}^L \|V^j\|_{L^6}^6+\|R_n^L\|_{L^6}^6+o_n(1),\label{Sum-Norm-L6}\\
  \|g_n\|_{\dot H^1}^2&=\sum\limits_{j=1}^L \|V^j\|_{\dot H^1}^2+\|R_n^L\|_{\dot H^1}^2+o_n(1)\label{Sum-Norm-H1}.
\end{align}
From \eqref{Sum-Norm-L4}--\eqref{Sum-Norm-H1}, we have
\begin{align}
  S(g_n)&=\sum\limits_{j=1}^L S(V^j)+S(R_n^L)+o_n(1),\label{Sum-Norm-Sc}\\
  K(g_n)&=\sum\limits_{j=1}^L K(V^j)+K(R_n^L)+o_n(1).\label{Sum-Norm-Kc}
\end{align}

Now we need the following lemma.
\begin{lem}\label{lem:kc-pro}
Let $f\in \H1\setminus \{0\}$, suppose that
\begin{align}\label{19.31}
\|f_x\|_{L^2}^2+\frac{1}{8}\|f\|_{L^4}^4\le d,
\end{align}
then $K(f)  \ge  0$.
\end{lem}
\begin{proof}
We prove by contradiction. Assume that there exists a function $f\in \H1\setminus \{0\}$ satisfies \eqref{19.31}, but $K(f)< 0$.
Then for
$$
\lambda=\frac{\sqrt 6\|f\|_{L^4}^2}{\|f\|_{L^6}^3},
$$
we have $\lambda < 1$ and $K(\lambda f)=0$. Then from the definition of $d$, we have $S(\lambda f)\ge d$.
However,
$$
\|\lambda f_x\|_{L^2}^2+\frac{1}{8}\|\lambda f\|_{L^4}^4=S(\lambda f)-\frac1{16}K(\lambda f)\ge d.
$$
Since $\lambda<1$, this contradicts with \eqref{19.31}. Thus we obtain the lemma.
\end{proof}

Now we finish the proof of proposition.
We first observe that
\begin{align}\label{lim-px-gn}
\|\partial_x g_n\|_{L^2}^2+\frac{1}{8}\|g_n\|_{L^4}^4\to d,\quad \mbox{ as } n\to \infty.
\end{align}
Since $S(g_n)\to d,  K(g_n)\to 0$,  we obtain
$$
\|\partial_x g_n\|_{L^2}+\frac{1}{8}\|g_n\|_{L^4}^4=S(g_n)-\frac1{16}K(g_n)\to d.
$$
Moreover, by \eqref{Sum-Norm-L4} and \eqref{Sum-Norm-H1}, we have
\begin{align}\label{sum-px-gn}
  \|\partial_x g_n\|_{L^2}^2+\frac{1}{8}\|g_n\|_{L^4}^4&=\sum\limits_{j=1}^L\big(\|\partial_x V^j\|_{L^2}^2+\frac{1}{8} \|V^j\|_{L^4}^4\big)+\big(\|\partial_x R_n^L\|_{L^2}^2+\frac{1}{8} \|R_n^L\|_{L^4}^4\big)+o_n(1).
\end{align}
Taking the limits $\lim\limits_{L\to\infty}\lim\limits_{n\to\infty}$ on both sides, we have
\begin{align*}
  \|\partial_x V^j\|_{L^2}^2+\frac{1}{8} \|V^j\|_{L^4}^4\le d, \quad \mbox{ for any } j=1,2,\cdots.
\end{align*}
Thus, by Lemma \ref{lem:kc-pro}, we have
\begin{align}\label{Kc-j}
  K(V^j)\ge 0, \quad \mbox{ for } j=1,2,\cdots.
\end{align}

Now taking the limits $\lim\limits_{L\to\infty}\lim\limits_{n\to\infty}$ on the both two sides of \eqref{Sum-Norm-Kc}, and by \eqref{Kc-j} we have
\begin{align}\label{Kc-j-2}
  K(V^j)=0, \quad \mbox{ for any } j=1,2,\cdots.
\end{align}
Then by the definition of $d$, we deduce that for any $j=1,2,\cdots,$
$$
\mbox{ either  } S(V^j)\ge d,\quad \mbox{ or  }\quad   V^j=0.
$$
Now we claim that there exists exactly one $j$, say, $j=1$ such that
$$
S(V^1)= d,
$$
and $V^j=0$ for other $j\ge 2$. Otherwise, $V^j=0$ for any $j\ge 1$. Then
$$
\|g_n\|_{\dot H^1}\to d, \quad \|g_n\|_{L^4\cap L^6}\to 0,\quad \mbox{ as } n\to \infty.
$$
This yields $\|g_n\|_{L^\infty}\to 0$  as $ n\to \infty.$ However, $K(g_n)\le 0$, which leads $g_n\equiv0$ by the same argument as in \eqref{K-0}.
This is contradicted with $S(g_n)\to d$ and $d>0$.
Since $K(V^1)= 0$, we obtain the minimizer $G=V^1$ which attains $d$.
Moreover, from \eqref{lim-px-gn} and  \eqref{sum-px-gn}, we find that the remainder term $R_n$ (since $L=1$, we may omit the superscript $L$),
$$
\lim\limits_{n\to \infty}\|R_n\|_{L^4\cap \dot H^1}=0.
$$
Thus we close the proof of the proposition.
\end{proof}
\noindent As mentioned before, $d\le S(W)$. In fact, we have the equality.
\begin{lem}\label{lem:variational}
$d=S(W)$.
\end{lem}
\begin{proof}
Consider the set
\begin{equation}\label{dc}
\mathcal{M}:=\{\phi\in \H1\setminus \{0\}:S(\phi)=d, K(\phi)=0\}.
\end{equation}
Then by Proposition 2.1, $\mathcal M \ne \emptyset$.
By the Lagrangian multiplier, there exists $\lambda$, such that for
any $\phi\in \mathcal{M}$, such that
\begin{equation}\label{s2.1:eqs}
S'(\phi)=\lambda K'(\phi).
\end{equation}
Testing a function $\psi=\frac12\phi-x\phi_x$, we have
\begin{equation}\label{s2.1:eqs1}
S'(\phi)\psi=\lambda K'(\phi)\psi.
\end{equation}
On one hand, since $S'(\phi)=2(-\partial_{xx} \phi+|\phi|^2\phi-\frac3{16}|\phi|^4\phi)$, we have
\begin{align*}
S'(\phi)\psi&=2\mbox{Re}\int_\R (-\phi_{xx}+|\phi|^2\phi-\frac{3}{16}|\phi|^4\phi)\bar{\psi}\,dx\\
&=\frac18\int_\R(6 |\phi|^4-|\phi|^6)\,dx=\frac18K(\phi).
\end{align*}
Thus $S'(\phi)\psi=0$ for any $\phi\in \mathcal{M}$.
On the other hand, $K'(\phi)=24|\phi|^2\phi-6|\phi|^4\phi$, gives that
\begin{align*}
K'(\phi)\psi&=6\mbox{Re}\int_\R (4|\phi|^2\phi-|\phi|^4\phi)\bar{\psi}\,dx\\
&=\int_\R(18|\phi|^4-4|\phi|^6)\,dx=-6\int_\R |\phi|^4\,dx.
\end{align*}
Thus, for any $\phi\in \mathcal{M}$, $K'(\phi)\psi=-6\int_\R |\phi|^4\,dx\neq 0$.
Therefore, from \eqref{s2.1:eqs1}, we obtain that $\lambda=0$.  Thus, \eqref{s2.1:eqs1} yields that $S'(\phi)=0$.
Hence, using Lemma 2.1, we obtain that $\phi=e^{i\theta}W(\cdot-x_0)$ for some $\theta,x_0\in\R$ and thus $d=S(W)$.
This proves the lemma.\end{proof}
\noindent From the proof of Lemma \ref{lem:variational}, we obtain
\begin{align*}
\mathcal{M} &= \{\phi\in \H1\setminus \{0\}:S(\phi)=S(W), K(\phi)=0\} \\
&=\{e^{i\theta}W(\cdot-x_0):\theta\in \R, x_0\in\R\}.
\end{align*}
This rigidity implies that the function $G$ obtained in Proposition \ref{prop:SK-minimizer} is equal to $e^{i\theta}W(\cdot-x_0)$ for some $\theta,x_0\in\R$. Therefore, we conclude Proposition \ref{prop:W-rigidity} from Proposition \ref{prop:SK-minimizer}.
\vspace{0.5cm}
\section{Proof of Theorem \ref{thm:main} and Theorem \ref{thm:main2}}
We first prove Theorem~\ref{thm:main}. Instead of proving Theorem \ref{thm:main}, we give a slightly more general result. This will be more useful in the proof of Theorem \ref{thm:main2}. To this end, we study the solution to \eqref{eqs:DNLS-under-gauge1}. We rewrite conservation laws in terms of $v(t)$ variable.
\begin{align}\label{mass-43}
M(v(t))&:=\|v(t)\|_{L^2_x}^2=
M(v_0),  \\
\label{mom-law}
P(v(t))&:=\frac12\mbox{Im} \int_{\R}\bar v(t) v_x(t)\,dx+\frac 18\int_{\R}
|v(t)|^4\,dx=P(v_0),  \\
\label{energy-43}
E(v(t))&:=\frac12\|v_x(t)\|_{L^2_x}^2-\frac{1}{32}\|v(t)\|^6_{L^6_x}=
E(v_0).
\end{align}

\begin{thm}\label{thm:general}
For any $\varepsilon>0$, there exists a $\delta=\delta(\varepsilon)$ such that if
\begin{align}\label{assump:intial2}
E(v_0)=O(\delta), \quad P(v_0)=O(\delta),\quad \mbox{ and } \quad M(v_0)=M(W)+O(\delta),
\end{align}
then the result in Theorem \ref{thm:main} holds.
\end{thm}
\noindent It is obvious that \eqref{assump:intial} implies \eqref{assump:intial2}. Hence Theorem \ref{thm:main} is a consequence of Theorem \ref{thm:general}. \\

Further, we fix a time $t$, and assume that $\|u\|_{L^6}^6\ge \sqrt\delta$, otherwise the first claim of Theorem \ref{thm:main} holds.
To simplify notations regarding to the functional $S$,
we set
$$
E_0=E(v_0),\quad P_0=P(v_0), \quad M_0=M(v_0).
$$
Then under the assumption \eqref{assump:intial2}, we have
\begin{align}\label{assump:EPM}
E_0, P_0=O(\delta),\quad \mbox{ and } M_0=M(W)+O(\delta),
\end{align}
where $O(\delta)\to 0$ as $\delta\to 0$.
\\
We define the function $w$ by
\begin{equation}\label{w}
w(t,x):=e^{-it+ix}v(t,x-2t), \qquad w_0=e^{ix}v_0.
\end{equation}
Then the assumption \eqref{assump:intial} becomes
\begin{align}\label{assump:intial-w0}
\|w_0-W\|_{H^1}\le \delta.
\end{align}
Again, we can rewrite conservation laws in $w(t,x) $ variable. The mass, momentum and energy conservation laws \eqref{mass-43}--\eqref{energy-43} are changed as follows:
\begin{align}
{M}(w(t))&:=\|w(t)\|_{L^2_x}^2=M_0,\label{mass-w}  \\
\label{mom-w}
\widetilde{P}(w(t))&:=\mbox{Im} \int_{\R}\bar w(t) w_x(t)\,dx-\|w(t)\|_{L^2_x}^2+\frac 14\int_{\R}
|w(t)|^4\,dx=P_0, \\
\label{energy-w}
\widetilde{E}(w(t))&:=\|w_x(t)\|_{L^2_x}^2-2\mbox{Im}\int \bar w(t)w_x(t)\,dx+\|w(t)\|_{L^2_x}^2 -\frac{1}{16}\|w(t)\|^6_{L^6_x}=
E_0.
\end{align}
We also find that
$$
S(w)=\widetilde{E}(w(t))+2\widetilde{P}(w(t))+\widetilde{M}(w(t))=E_0+2P_0+M_0.
$$
Thus by \eqref{assump:EPM}, we have
\begin{align}\label{assump:Sw}
S(w(t))&=M(W)+O(\delta)\notag\\
&=S(W)+O(\delta).
\end{align}
\noindent Now we consider the relationship between $\|v(t)\|_{L^4}$ and $\|v(t)\|_{L^6}$.
We denote
$$
f(t)=\frac{\|v(t)\|_{L^4}^4}{\|v(t)\|_{L^6}^3}.
$$
We first prove that
\begin{prop}\label{prop:f} For any  $t\in I$, if $\|u\|_{L^6}^6\ge \sqrt\delta$, then
$$
\big|f(t)^2-\frac{8}{3}\pi\big|\le O(\delta).
$$
\end{prop}
\noindent To prove this proposition, we adopt the argument in \cite{Wu2}. We sketch the proof when the argument is highly similar to \cite{Wu2}. Firstly, we have
\begin{lem}\label{lem:f-bound} For any $t\in I$,
$$
2C_{GN}^{-\frac92}+O(\delta)\le f(t)\le \sqrt{M_0}.
$$
\end{lem}
\begin{proof}
From the H\"older's inequality, we have
$$\|v(t)\|_{L^4}^4\le \|v(t)\|_{L^2}\|v(t)\|_{L^6}^3=\sqrt{M_0}\|v(t)\|_{L^6}^3,$$
and thus
$$
f(t)\le \sqrt{M_0}.
$$
On the other hand, by using the similar argument in \cite{Wu2}, we have
\begin{align*}
f(t)&\ge2C_{GN}^{-\frac92}+\varepsilon(t),
\end{align*}
where
$$
\varepsilon(t):=2C_{GN}^{-\frac92}\frac{\|v(t)\|_{L^6}^{\frac32}-\Big(\|v(t)\|_{L^6}^6+16E_0\Big)^{\frac14}}{\Big(\|v(t)\|_{L^6}^6+16E_0\Big)^{\frac14}}
.
$$
By the Mean Value Theorem, $E_0=C\delta$,  and $\|v(t)\|_{L^6}^6\ge \sqrt\delta$, we have
$$
\varepsilon(t)=CE_0\|v(t)\|_{L^6}^{-6}=O(\delta).
$$
This proves the lemma.
\end{proof}

\begin{proof}[Proof of Proposition \ref{prop:f}] $ $\\
We define
$$
\phi(t,x)= e^{i\alpha x}v(t,x),
$$
where the parameter $\alpha$ depends on $t$, and  is given below. Then we have
\begin{align*}
E(\phi)= E(v)+2\alpha\mbox{Im}\int \bar v\>v_x\,dx+\alpha^2\|v\|_{L^2 }^2.
\end{align*}
By the mass, energy conservation laws \eqref{mass-43} and \eqref{energy-43}, \eqref{sharp Gagliardo-Nirenberg2}, we have for any $\alpha >0 $,
\begin{align*}
-2\alpha\mbox{Im}\int \overline{v(t,x)}\>v_x(t,x)\,dx\le \Big(\frac1{16}-C_{GN}^{-18}f(t)^{-4}\Big)\|v(t)\|_{L^6 }^6+\alpha^2M_0+E_0,
\end{align*}
or
\begin{align}\label{15.56}
-\mbox{Im}\int \overline{v(t,x)}\>v_x(t,x)\,dx\le \frac1{2\alpha}\Big(\frac1{16}-C_{GN}^{-18}f(t)^{-4}\Big)\|v(t)\|_{L^6 }^6+\frac12\alpha M_0+ \frac1{2\alpha}E_0.
\end{align}
By the momentum conservation law \eqref{mom-law}, we estimate
\begin{align}\label{15.57}
\frac14\|v(t)\|_{L^4}^4\le \frac1{2\alpha}\Big(\frac1{16}-C_{GN}^{-18}f(t)^{-4}\Big)\|v(t)\|_{L^6 }^6+\frac12\alpha M_0+ \frac1{2\alpha}E_0+P_0.
\end{align}
Next, we claim that for any $t\in I$,
\begin{align}\label{15.58}
\Big(\frac1{16}-C_{GN}^{-18}f(t)^{-4}\Big)\|v(t)\|_{L^6 }^6\ge |E_0|.
\end{align}
To prove \eqref{15.58}, for a contradiction, we assume there exists a time $t_0$ such that the negation of \eqref{15.58} holds. Then choosing $\alpha=\sqrt{|E_0|}$, we have
$$
\frac14\|v(t_0)\|_{L^4}^4\le \sqrt{|E_0|}+P_0= O(\delta).
$$
But by Lemma \ref{lem:f-bound}, $\|v(t)\|_{L^4}^4$ is on the level of $\sqrt \delta$, so suitably narrowing $\delta$, we reach the contradiction.
\\
Now, we choose $$
\alpha(t)= \sqrt{M_0^{-1}\Big(\frac1{16}-C_{GN}^{-18}f(t)^{-4}\Big)}\|v(t)\|_{L^6}^3.
$$
By \eqref{15.57} and \eqref{15.58}, we estimate $\alpha(t)\ge  \sqrt{M_0^{-1}|E_0|}$ and
\begin{align}
\|v(t)\|_{L^4}^4\le &\sqrt{M_0\Big(1-16C_{GN}^{-18}f(t)^{-4}\Big)}\|v(t)\|_{L^6 }^3+ 2\alpha^{-1}E_0+4P_0;\notag\\
\le &\sqrt{M_0\Big(1-16C_{GN}^{-18}f(t)^{-4}\Big)}\|v(t)\|_{L^6 }^3+O(\delta). \label{Key-1}
\end{align}
Since $\|v(t)\|_{L^6 }^6\ge \sqrt\delta$, by \eqref{Key-1}, we find that
\begin{align*}
f\le \sqrt{M_0\Big(1-16C_{GN}^{-18}f^{-4}\Big)}+O(\delta).
\end{align*}
By Lemma~\ref{lem:f-bound}, we obtain
\begin{align}\label{inequality1}
f^6\le M_0f^4-16M_0C_{GN}^{-18}+O(\delta).
\end{align}
Note that the equation
$$
X^3-M_0X^2+16M_0C_{GN}^{-18}\le 0
$$
admits only one solution $X=\frac83 \pi$ when $M_0=4\pi$. Thus when $M_0=4\pi+O(\delta)$,  by the continuity argument, we have $f^2=\frac83 \pi+O(\delta)$. \end{proof}
Now we use the scaling argument, let $\lambda(t)=\|W\|_{L^6}/\|v(t)\|_{L^6}$, and
$$
v_\lambda(t,x)=\lambda^{\frac12}v(\lambda^2 t,\lambda x).
$$
Then  $\lambda\le \delta^\frac1{12}\|W\|_{L^6}$, and
\begin{align}
\|v_\lambda(t)\|_{L^6}^6=&\|W\|_{L^6}^6=96\pi;\label{L6-scaling}.
\end{align}
Since $f(t)$ is scaling invariant, i.e. $\frac{\|v_\lambda(t)\|_{L^4}^4}{\|v_\lambda(t)\|_{L^6}^3}=f(t)$, we have
\begin{align}
\|v_\lambda(t)\|_{L^4}^4=&f(t)\|W\|_{L^6}^3=16\pi+O(\delta)=\|W\|_{L^4}^4+O(\delta).\label{L4-scaling}
\end{align}
\noindent Let $w(t,x;\lambda)$ be defined as
\begin{align}\label{w-lambda}
w(t,x;\lambda):=e^{-it+ix}v_\lambda(t,x-2t).
\end{align}
Then
\begin{align}
M\big(w(\lambda)\big)=&M(v_\lambda)=M(v)=M_0;\\
\widetilde P\big(w(\lambda)\big)=&P(v_\lambda)=\lambda P(v)=\lambda P_0;\\
\widetilde E\big(w(\lambda)\big)=&E(v_\lambda)=\lambda^2 E(v)=\lambda^2 E_0.
\end{align}
From $\widetilde P$ in \eqref{mom-w}, we have
\begin{align*}
\mbox{Im} \int_{\R}\bar w(t;\lambda) \partial_xw(t;\lambda)\,dx=\|w(t;\lambda)\|_{L^2_x}^2-\frac 14\int_{\R}
|w(t;\lambda)|^4\,dx+\lambda P_0.
\end{align*}
Note that  $\lambda\le \varepsilon_0^{-1}\|W\|_{L^6}$, combining this with \eqref{L4-scaling}, we have
\begin{align}\label{eqs:Im}
\mbox{Im} \int_{\R}\bar w(t;\lambda) \partial_xw(t;\lambda)\,dx= O(\delta).
\end{align}
Inserting \eqref{eqs:Im} into $\widetilde E(w_\lambda)$ in \eqref{energy-w},  and applying \eqref{L6-scaling}, we have
\begin{align}\label{eqs:H1-w}
\|\partial_xw(t;\lambda)\|_{L^2_x}^2=&2\mbox{Im}\int \bar w(t;\lambda)\partial_xw(t;\lambda)\,dx-\|w(t;\lambda)\|_{L^2_x}^2 +\frac{1}{16}\|w(t;\lambda)\|^6_{L^6_x}+
\lambda^2E_0;\notag\\
= & 2\pi+O(\delta).
\end{align}
\noindent Therefore, by \eqref{L6-scaling}, \eqref{L4-scaling} and \eqref{eqs:H1-w}, we have
\begin{align*}
S(w(\lambda))=&\|\partial_xw(\lambda)\|_{L^2}^2+\frac12\|w(\lambda)\|_{L^4}^4-\frac{1}{16}\|w(\lambda)\|_{L^6}^6\\
=&4\pi+O(\delta)=S(W)+O(\delta);\\
K(w(\lambda))=&6\|w(\lambda)\|_{L^4}^4-\|w(\lambda)\|_{L^6}^6=O(\delta).
\end{align*}
By Proposition~\ref{prop:W-rigidity}, we get
$$
\inf\limits_{(\theta,y)\in\R^2} \big\|w(\lambda)-e^{i\theta}W(\cdot-y)\big\|_{\dot H^1}<\varepsilon.
$$
By the mass conservation law, we further obtain
$$
\inf\limits_{(\theta,y)\in\R^2} \big\|w(\lambda)-e^{i\theta}W(\cdot-y)\big\|_{ H^1}<\varepsilon.
$$
Thus, by \eqref{w-lambda} and \eqref{gauge1}, we prove that
$$
\|u(t)-e^{i\theta}R_{\lambda^{-1}}(t,\cdot-y)\|_{H^1}\le \varepsilon.
$$
This proves Theorem \ref{thm:general}.

\subsection*{Proof of Theorem \ref{thm:main2}} $ $\\
Set $\lambda(s)=\|\partial_x W\|_{L^2}/\|\partial_x v(s)\|_{L^2}$, and
\begin{align}
v_{[s]}(t,x)=\lambda(s)^{\frac12}v\big(\lambda(s)^2 t,\lambda(s) x\big). \label{def:vs}
\end{align}
Then by the conservation laws of $v$, we have
\begin{align*}
E(v_{[s]}(t))=&E(v_{[s]}(0))=\lambda(s)^2E(v_0),\\
P(v_{[s]}(t))=&P(v_{[s]}(0))=\lambda(s)^2P(v_0),\\
M(v_{[s]}(t))=&M(v_{[s]}(0))=M(v_0).
\end{align*}
So under the assumption of Theorem \ref{thm:main2}, if the solution $v$ to \eqref{eqs:DNLS-under-gauge1} blows up in the finite time $T^*$, then $\|\partial_x v(s)\|_{L^2}\to \infty$ as $s\to T^*$. Thus,
we have
$$\lambda(s)\to 0, \mbox{ as } s\to T^*.$$
Hence, when $s\to T^*$,
$$
E(v_{[s]}(0)), P(v_{[s]}(0))\to 0, \mbox{ and } M(v_{[s]}(0))=M(W).
$$
This implies from Theorem \ref{thm:general} that  there exist $\theta_s(t)\in
[0,2\pi), y_s(t)\in\R$, and $\widetilde \lambda_s(t)\in [ \lambda_0,\infty)$ such that when $s\to T^*$,
$$
u_{[s]}(t)-e^{i\theta_s(t)}R_{\widetilde\lambda_s(t)}(t,\cdot-y_s(t))\to 0 \quad \mbox{in }\H1,
$$
uniformly in $t$.
Moreover,  $\widetilde\lambda_s(s)=1$. In particular, when $t=s$, we have
$$
u_{[s]}(s)-e^{i\theta_s(s)}R(s,\cdot-y_s(s))\to 0 \quad \mbox{in }\H1.
$$
In view of \eqref{def:vs}, we finish the proof of Theorem \ref{thm:main2}.

\section*{Acknowledgements}
The authors are very grateful to Masayuki Hayashi who pointed out a mistake in Proposition 1 in the previous version.


\begin{thebibliography}{99}

\bibitem{Agueh}
Agueh, M.,
Sharp Gagliardo-Nirenberg Inequalities and Mass
Transport Theory, J.\ Dyn.\ Differ.\ Equ., 18, 1069--1093 (2006).




\bibitem{BeLi-1983}
Berestycki, H. and Lions,  P. L., Nonlinear scalar field equations, I, existence of a ground state, Arch.\ Rational Mech.\
Anal., 82, 313--345 (1983).




\bibitem{CoOh-06-DNLS}Colin, M. and Ohta, M., {Stability of solitary waves for derivative nonlinear Schrodinger equation},
Ann. I. H. Poincar\'{e}-AN, 23, 753--764 (2006).

\bibitem{CKSTT-01-DNLS} Colliander, J.; Keel, M.; Staffilani, G.;  Takaoka, H.; and Tao, T., {Global well-posedness result for Schr\"{o}dinger equations with derivative,} SIAM J.\ Math.\ Anal., { 33} (2), 649--669 (2001).

\bibitem{CKSTT-02-DNLS} Colliander, J.; Keel, M.; Staffilani, G.;  Takaoka, H.; and Tao, T., {A refined global well-posedness result for Schr\"{o}dinger equations with derivatives,} SIAM J.\ Math.\ Anal., { 34}, 64-86 (2002).

\bibitem{Garard-98}
G\'{e}rard, P., {Description of the lack
of compactness of a Sobolev embedding}, ESAIM Control Optim. Calc.
Var. 3, 213--233 (1998).



\bibitem{Grhe-95}
Gr\"{u}nrock, A.; Herr, S.,  Low regularity local well-posedness of the derivative nonlinear Schr\"odinger
equation with periodic initial data, SIAM J. Math. Anal. 39 (6), 1890--920 (2008).






\bibitem{GuWu95} Guo, Boling; and Wu, Yaping, Orbital stability of solitary waves for the nonlinear derivative Schr\"odinger equation, J. Differential Equations, 123, 35--55 (1995).


\bibitem{Ha-93-DNLS}
Hayashi, N., {The initial value problem for the derivative
nonlinear Schr\"{o}dinger equation in the energy space,} Nonl.
Anal., {20}, 823--833 (1993).

\bibitem{HaOz-92-DNLS}
Hayashi, N.; and  Ozawa, T., {On the derivative nonlinear
Schr\"{o}dinger equation,} Physica\ D., {55}, 14--36 (1992).

\bibitem{HaOz-94-DNLS}
Hayashi, N.; and   Ozawa, T., {Finite energy solution of
nonlinear Schr\"{o}dinger equations of derivative type,} SIAM J.
Math. Anal., {25}, 1488--1503 (1994).



\bibitem{Herr}
Herr, S., On the Cauchy problem for the derivative nonlinear Schr\"odinger equation with periodic boundary
condition, Int. Math. Res. Not. 2006, Art. ID 96763, 33 (2006).





\bibitem{Miao-Wu-Xu:2011:DNLS}
Miao, Changxing; Wu, Yifei; and Xu, Guixiang, {Global well-posedness for
Schr\"{o}dinger equation with derivative in $H^{\frac 12} (\R)$,}
J.\ Diff.\ Eq., 251, 2164--2195 (2011).





\bibitem{NaOhSt-12}
Nahmod, A. R.; Oh, T.; Rey-Bellet, L.; and Staffilani,  G., Invariant weighted Wiener measures and almost
sure global well-posedness for the periodic derivative NLS, J. Eur. Math. Soc. 14 (4), 1275--1330 (2012).


\bibitem{Ohta-14} Ohta, M., Instability of solitary waves for nonlinear Schr\"odinger equations of derivative type, arXiv:1408.5537.

\bibitem{Oz-96-DNLS}
Ozawa, T.,  {On the nonlinear Schr\"{o}dinger equations of
derivative type,} Indiana Univ.\ Math.\ J., {45}, 137--163 (1996).



\bibitem{Ta-99-DNLS-LWP}
Takaoka, H., {Well-posedness for the one dimensional
Schr\"{o}dinger equation with the derivative nonlinearity,} Adv.\
Diff.\ Eq., {4 }, 561--680 (1999).

\bibitem{Ta-01-DNLS-GWP}
Takaoka, H., {Global well-posedness for Schr\"{o}dinger
equations with derivative in a nonlinear term and data in low-order
Sobolev spaces,} Electron.\ J.\ Diff.\ Eqns., {42}, 1--23 (2001).



\bibitem{W}
Weinstein, M. I., {Nonlinear {S}chr\"odinger equations and sharp
interpolation estimates,} Comm. Math. Phys., { 87}(4), 567--576
(1983).


\bibitem{Wu1}
Wu, Yifei,  Global well-posedness of the derivative nonlinear Schr\"odinger equations in energy space, Analysis \& PDE, 6 (8), 1989--2002 (2013).


\bibitem{Wu2}
Wu, Yifei,  Global well-posedness on the derivative nonlinear Schr\"odinger
equation, preprint, arXiv:1404.5159.


\end{thebibliography}
\end{document}